\documentclass[12pt,reqno]{elsarticle}

\usepackage{amsfonts,color,amsmath,amssymb,fancyhdr}
\usepackage{amsthm}

\def\Box{\vcenter{\vbox{\hrule\hbox{\vrule
     \vbox to 8.8pt{\hbox to 10pt{}\vfill}\vrule}\hrule}}}

\newcommand{\F}{\mathbb{F}}

\newcommand{\bS}{\mathbb{S}}
\newcommand{\tr}{\textup{Tr}}
\newtheorem{thm}{Theorem}[section]
\newtheorem{lemma}[thm]{Lemma}

\numberwithin{equation}{section}
\begin{document}
\newcommand{\stopthm}{\begin{flushright}
\(\box \;\;\;\;\;\;\;\;\;\; \)
\end{flushright}}

\newcommand{\symfont}{\fam \mathfam}

\title{On the isotopism classes of the Budaghyan-Helleseth commutative semifields}
\author[add1]{Tao Feng} \ead{tfeng@zju.edu.cn}
\author[add1]{Weicong Li\corref{cor1}}\ead{conglw@zju.edu.cn}
\cortext[cor1]{Corresponding author}
\address[add1]{School of Mathematical Sciences, Zhejiang University, 38 Zheda Road, Hangzhou 310027, Zhejiang P.R China}

\begin{abstract}
In this paper, we completely determine the isotopism classes of  the Budaghyan-Helleseth commutative semifields constructed in [L. Budaghyan, T. Helleseth, New commutative semifields defined by PN multinomials, Crypto. Comm. 3 (1), 2011, p. 1-16].
 \vspace*{3mm}

\noindent \textbf{Keyword:} commutative semifield, Budaghya-Hellsseth family, isotopism, strong isotopism.
\end{abstract}

\maketitle

\section{Introduction}

A finite \textit{presemifield} $\bS$ is a finite ring with no zero-divisors such that both the left and right distributive laws hold. If it further contains a multiplicative identity, then we call $\bS$ a {\it semifield}. A semifield is not necessarily commutative or associative, but by Wedderburn's Theorem \cite{Maclagan1905A}, associativity necessarily implies commutativity in the finite case. The study of semifields was initiated by Dickson \cite{Dickson1906Linear} in the study of division algebras. Knuth \cite{Knuth1965Finite} showed that the additive group of a semifield $\bS$ is an elementary abelian group. Hence, any finite presemifield can be represented by  $(\F_{p^n},\,+,\,\ast)$, where $(\F_{p^n},\,+)$ is the additive group of the finite field $\F_{p^n}$ with $p^n$ elements and $x\ast y= \varphi(x,y)$ with $\varphi$ a bilinear function from $\F_{p^n}\times \F_{p^n}$ to $\F_{p^n}$. For a recent survey on finite semifields, please refer to \cite{lavrauw2011finite}. In this paper, we are concerned with commutative presemifields with odd characteristic. Such presemifields can be equivalently described by planar polynomials of Dembowski-Ostrom type, cf. \cite{Coulter2008Commutative,Coulter1997Planar}.

Let $\bS_1=(\F_{p^n},+,\ast)$ and $\bS_2=(\F_{p^n},\,+,\,\ast')$ be two presemifields. They are {\it isotopic} if there exist three linear permutations $L,\,M,\,N$ over $\F_{p^n}$ such that $L(x\ast y)=M(x)\ast' N(y)$ for all $x,y \in \F_{p^n}$, and we say that $(M,\,N,\,L)$ is an {\it isotopism} between $\bS_1$ and $\bS_2$. If there is an isotopism $(N,\,N,\,L)$ between the presemifields $\bS_1$ and $\bS_2$, then it is a {\it strong isotopism} and $\bS_1$ and $\bS_2$ are {\it strongly isotopic}. In the particular case $\bS_1=\bS_2$, a (strong) isotopism is called a (strong) {\it autotopism}. Both isotopism and strong isotopism define an equivalence relation on the finite presemifields, and the equivalence classes are called  isotopism classes and strong isotopism classes respectively. Two isotopic presemifields are called {\it isotopes} of each other.   For a presemifield $ \bS=(\F_{p^n},+,\ast)$ and its nonzero element $e$, we define a new multiplication $\star$ by $(x\ast e)\star (e\ast y)=x\ast y$.
Then $\bS'=(\F_{p^n},+,\star)$ is a semifield with an identity $e\ast e$, which is strongly isotopic to $\bS$.

In \cite{Budaghyan2008New,Budaghyan2011New}, Budaghyan and Hellsseth constructed two families of commutative presemifileds of order $p^{2k}$ from certain planar functions of Dembowski-Ostrom type over $\F_{p^{2k}}$, where $p$ is an odd prime. They established that the first family is non-isotopic to previously known semifields for $p\neq 3$ and $k$ odd, and determined the middle nuclei  in some special cases. Later on, Bierbrauer \cite{Bierbrauer2011Commutative} observed that the two families are in fact the same and have been  also independently discovered in \cite{Zha2009New}. This family of semifields is now commonly called the Budaghyan-Hellsseth family in the literature. In the same paper, Bierbrauer gave a generalization of the Budaghyan-Hellsseth family which also contains the LMPT-construction \cite{Lunardon2011Symplectic}. The new semifields sometimes are called the LMPTB family, and it is not isotopic to any previously known commutative semifields with the possible exception of the Budaghyan-Hellsseth family as shown in \cite{Bierbrauer2011Commutative}. However, Marino and Polverino  proved that the LMPTB family is contained in the Budaghyan-Hellsseth family in \cite{Marino2011On}, and they determined the nuclei and middle nuclei of the Budaghyan-Hellsseth semifields in \cite{Marino2011On}.

In this paper, we completely determine the isotopism classes of the Budaghyan-Hellsseth presemifields. In Section 2, we introduce some background and preliminary results. We determine the strong isotopism classes of the Budaghyan-Hellsseth presemifields in Section 3, and completely determine the isotopism classes in Section 4. For odd $q$, with the center size $q$ and semifield order $q^{2l}$ fixed, the number of isotopism classes in the Budaghyan-Hellsseth family is equal to
\begin{enumerate}
\item[(i)] $\phi(l)/2$ in the case $q\equiv1\pmod{4}$ and $l>2$ is even;
\item[(ii)] $\phi(l)$ in the case $q\equiv3\pmod{4}$ and $l>2$ is even;
\item[(iii)] $\phi(l)/2$ in the case $l$ is odd;
\end{enumerate}
Here, $\phi$ is the Euler totient function. In the case $l=2$, it is isotopic to a Dickson semifield.

\section{Preliminaries}

We start by introducing some notation that we shall use throughout the paper. Let $p$ be an odd prime, and let $l,\,h,\,d$ be positive integers such that
\[
 1<l,\quad 1\le d\le 2lh-1,\quad \gcd(l,d)=1,\quad\textup{and $l+d$ is odd}.
\]
Set $q=p^h$, and fix a nonsquare $\beta$ of $\F_{q^{2l}}$ and a nonzero element $\omega \in\F_{q^{2l}}$ such that $\omega+ \omega^{q^l}=0$. Let $\tr:\,\F_{q^{2l}}\rightarrow\F_{q^l}$ be the trace function $\tr(x)=x+x^{q^l}$.  We define the following multiplication
\begin{equation}
   x\ast_{(d,\beta)} y= x^{q^l}y+xy^{q^l} + \left(\beta(x^{q^d}y+xy^{q^d})+\beta^{q^l}(x^{q^d}y+xy^{q^d})^{q^l}\right)\,\omega. \label{2.1}
\end{equation}
Then $(\F_{q^{2l}},\,+,\,\ast_{(d,\beta)})$ is a presemifield in the Budaghyan-Helleseth family. This is the simplified version as found in \cite{Marino2012On}. It is clear that different choices of $\omega$ lead to strongly isotopic presemifields. We now show that different choices of $\beta$ also lead to strongly isotopic presemifields.

\begin{lemma}\label{lem_b0b1}
Let $\beta$ and $\beta'$ be two nonsquares in $\F_{q^{2l}}$, and let $d$ be an integer such that $\gcd(d,\,l)=1$ and $l+d$ is odd. Then $\gcd(q^d+1,\,q^l+1)=2$, and there exists nonzero elements $b_0,\,b_1\in\F_{q^{2l}}$ such that
\[
\beta'\beta^{-1}b_0^{q^d+1}\in\F_{q^l},\quad \beta^{1-q^{2l-d}}b_1^{q^{2l-d}+1}\in\F_{q^l}.
\]
\end{lemma}
\begin{proof}
The fact $\gcd(q^d+1,\,q^l+1)=2$ has been already proved in Lemma 3.1 (ii) of \cite{Marino2012On}.
We fix a primitive element $\gamma$ of $\F_{q^{2l}}$ and write $\log(\gamma^i):=i\pmod{q^{2l}-1}$. The conditions translate to
\begin{align*}
\log(\beta')-\log(\beta)+(q^d+1)\log(b_0)&\equiv 0\pmod{q^l+1},\\
(1-q^{2l-d})\log(\beta)+(q^{2l-d}+1)\log(b_1)&\equiv 0\pmod{q^l+1}.
\end{align*}
Since $\log(\beta')$ and $\log(\beta)$ are both odd, the existence of the desired $b_0$ and $b_1$ follows.
\end{proof}
Take $\beta$, $\beta'$, $b_0$ and $b_1$ as in Lemma \ref{lem_b0b1}. Set $N_0(x)=b_0x$, $N_1(x)=b_1x$, and
\begin{align*}
L_0(x)&=\frac{1}{2} b_0^{q^l+1}(x+x^{q^l}) +\frac{1}{2}\beta' \beta^{-1}b_0^{q^d+1}(x-x^{q^l}),\\
L_1(x)&=\frac{1}{2} b_1^{q^l+1}(x+x^{q^l})+ \frac{1}{2} \beta^{1-q^{2l-d}}b_1^{q^{2l-d}+1}\omega^{1-q^{l-d}}(x-x^{q^l})^{q^{l-d}}.
\end{align*}
It is straightforward to check that
\begin{align*}
L_0(x\ast_{(d,\beta)} y)= &b_0^{q^l+1}\tr(x^{q^l}y)
+\beta'\beta^{-1}b_0^{q^d+1}\tr(\beta(x^{q^d}y+xy^{q^d}))\, \omega \\
  =&b_0^{q^l+1}\tr(x^{q^l}y)+\tr(\beta'b_0^{q^d+1}(x^{q^d}y+xy^{q^d}))\,\omega\\
  =& N_0(x)*_{(d,\beta')} N_0(y).
\end{align*}
Therefore, in the multiplication \eqref{2.1} different choices of $\beta$ lead to strongly isotropic presemifields. We thus fix a nonsquare $\beta$, write $*_d$ instead of $*_{(d,\,\beta)}$, and use the notation $BH(q,\,l,\,d)$ for $(\F_{q^{2l}},\,+,\,*_d)$. If $q$ and $l$ are fixed and clear from the context, we shall write $\bS_d:=(\F_{q^{2l}},\,+,\,\star_d)$ for the corresponding semifield with multiplication
\begin{equation}\label{2.2}
(1*_d x)\star_d(1*_dy)=x*_d y.
\end{equation}
\begin{lemma}\label{lem_2.2}
For $0<d<2lh-1$ such that $\gcd(l,d)=1$ and $l+d$ is odd, the presemifields $BH(q,\,l,\,d)$ and $BH(q,\,l,\,2l-d)$ are in the same strong isotopism class.
\end{lemma}
\begin{proof}
From $\omega+\omega^{q^l}=0$ we deduce that $\omega^{1-q^{l-d}}\in \F_{q^l}$. With $L_1$ and $N_1$ as defined above,
\begin{align*}
L_1(x\ast_d y) = &b_1^{q^l+1}\tr(x^{q^l}y)+
   \beta^{1-q^{2l-d}}b_1^{q^{2l-d}+1}\omega^{1-q^{l-d}}\tr\big( \beta^{q^{2l-d}}(x^{q^{l+d}}y^{q^l}+x^{q^l}y^{q^{l+d}})^{q^{l-d}}\big)\, \omega^{q^{l-d}}\\
  = &b_1^{q^l+1}\tr(x^{q^l}y)+
  \tr\big( \beta  b_1^{q^{2l-d}+1}(xy^{q^{2l-d}}+x^{q^{2l-d}}y) \big)\,\omega\\
  =&N_1(x)*_{2l-d} N_1(y).
\end{align*}
The triple $(N_1,\,N_1,\,L_1)$ is the desired strong isotopism.
\end{proof}

Let $\bS=(\F_{p^n},+,\star)$ be a commutative semifield. Its left nucleus $N_l(\bS)$ and middle nucleus $N_m(\bS)$ are defined as follows:
\begin{align*}
N_l(\bS)&=\{a\in \F_{p^n}:\,(a\star x)\star y=a\star(x\star y)\text{ for all $x,\,y$}\in \F_{p^n}\},\\
N_m(\bS)&=\{a\in \F_{p^n}:\,(x \star a)\star y=x\star(a\star y) \text{ for all $x,\,y$}\in \F_{p^n} \}.
\end{align*}
They are both finite fields, and $N_l(\bS)$ is also called the center of $\bS$ since for a commutative semifield the center and left nucleus coincide. Their sizes are invariants under isotopism. The explicit expressions of the nucleus and middle nucleus of the semifield $\bS_d$ corresponding to $BH(q,\,l,\,d)$ have been determined in \cite[Theorem 4]{Marino2011On}.
\begin{thm}\label{thm_Nm}\cite{Marino2011On}
Let $\bS_d:=(\F_{q^{2l}},\,+,\,\star_d)$ be the semifield with multiplication \eqref{2.2}. Then its center has size $q$, and its middle nucleus has size $q^2$.  For each $\alpha \in N_m(\bS_d)$, there exist $a,\,b\in\F_q$ such that
$(x\ast_d 1) \star_d \alpha = (ax+b\xi x^{q^l}) \ast_d 1$ for all $x\in\F_{q^{2l}}$, where $\xi$ is a constant such that $\beta^{1-q^l}=\xi^{q^{l+d}-1}$.
\end{thm}
Throughout the paper, we shall write $\alpha=\kappa(a,b)$, where $\alpha,\,a,\,b$ are defined as in Theorem \ref{thm_Nm}.

\begin{lemma}\label{lemma_bne0}
Take the same notation as in Theorem \ref{thm_Nm}. Then $\xi^{q^l+1}$ is a nonsquare in $\F_q^*$. For $a,\, b\in\F_q$, the element $\alpha=\kappa(a,\,b)$ is a nonsquare in  $N_m(\bS_d)$ if and only if $a^2-b^2\xi^{q^l+1}$ is a nonsquare in $\F_q$. In particular, if $\alpha$ is nonsquare, then $b\neq 0$.
\end{lemma}

\begin{proof}
From $\beta^{1-q^l}=\xi^{q^{l+d}-1}$, we deduce that $\xi^{(q^l+1)(q^{l+d}-1)}=1$. Since $\gcd(q^{l+d}-1,q^{l}-1)=q^{\gcd(l+d,l)}-1=q-1$, it follows that $\xi^{q^l+1}$ is in $\F_q$. Moreover, since $\beta$ is a nonsquare and $l+d$ is odd,
\[
-1=\xi^{(q^l+1)(q^{l+d}-1)/2}=\left(\xi^{(q^l+1)(q-1)/2}\right)^{(q^{l+d}-1)/(q-1)}=\xi^{(q^l+1)(q-1)/2}.
\]
It follows that $\xi^{q^l+1}$ is a nonsquare in $\F_q$. This proves the first claim.

For $\alpha=\kappa(a,\,b)$ in $N_m(\bS_d)$, if there exists $\gamma= \kappa(a_1,\,b_1)\in N_m(\bS_d)$ such that $\alpha=\gamma\star_d\gamma$, then by Theorem \ref{thm_Nm} $((x\ast_d 1) \star_d \alpha )= ((x\ast_d 1)\star_d \gamma)\star_d \gamma$ for all $x\in\F_{q^{2l}}$. We deduce that
\begin{align*}
(ax+b\xi x^{q^l}) \ast_d 1
 &= \left((a_1x+b_1\xi x^{q^l})\ast_d 1 \right) \star_d \gamma \\
 &= \left(a_1(a_1x+b_1\xi x^{q^l})+ b_1\xi (a_1x+b_1\xi x^{q^l})^{q^l}\right)\ast_d 1 \\
 &= \left((a_1^2+b_1^2\xi^{q^l+1}) x+2a_1b_1 \xi x^{q^l} \right)\ast_d 1.
\end{align*}
It follows that $a=a_1^2+b_1^2\xi^{q^l+1}$ and $b=2a_1b_1$. Therefore, $\alpha=\kappa(a,b)$ is a nonsquare in $N_m(\bS_d)$ if and only if
$X^2+Y^2\xi^{q^l+1}=a$, $2XY=b$ have no solutions in $\F_q\times\F_q$.	

If $b=0$, then there is always a solution since $\xi^{q^l+1}$ is a nonsquare. Now suppose that $b\ne 0$. We cancel out the variable $Y$ by using $Y=b/(2X)$ and get a quartic equation in $X$: $4X^4-4aX^2+b^2\xi^{q^l+1}=0$. Set $D:=a^2-b^2\xi^{q^l+1}$. We consider two cases.
\begin{enumerate}
\item[(i)] If $D$ is a nonsquare in $\F_q$, then the equation does not have a solution in $\F_q$.
\item[(ii)] If $D$ is a square in $\F_q$, then $4T^2-4aT+b^2\xi^{q^l+1}=0$ has two solutions in $\F_q$ whose product is a nonsquare $\frac{1}{4}b^2\xi^{q^l+1}$ in $\F_q$. One of the solutions is thus a square.
\end{enumerate}
This completes the proof.
\end{proof}

We recall the following characterization of Dickson semifields.
\begin{thm}\label{th_bound}\cite[Theorem 1.1]{Blokhuis2003On}
Let $\bS=(\F_{q^{2n}},+,*)$ be a commutative semifield with center $\F_q$ and middle nucleus $\F_{q^n}$, $q$ odd. If $q\geq 4n^2-8n+2$, then $\bS$ is either a Dickson semifield or a field.
\end{thm}
Applying it to the case $n=2$, we deduce that the presemifield $BH(q,\,2,\,1)$  must be isotopic to a Dickson semifield.The consequence have been already considered in Corollary 3 of \cite{Marino2011On}. Therefore, {\bf we shall only consider the case $l>2$ below.} Also, by Lemma \ref{lem_2.2}, we assume without loss of generality that $0<d< l$.

\section{Strong isotopisms among the Budaghyan-Hellsseth presemifields}
In this section, we consider the strong isotopisms among the presemifield $BH(q,l,d)$'s. 
The main result of this section is as follows.
\begin{thm}\label{thm_str}
If $0<d,\,d'\le l-1$, then $BH(q,\,l,\,d)$ and $BH(q,\,l,\,d')$ are strongly isotopic if and only if $d=d'$. Moreover, the strong autotopism group of $BH(q,\,l,\,d)$ has order $4lh(q^l-1)$.
\end{thm}

We split the proof  into several lemmas. Assume that $BH(q,l,d)$ and $BH(q,l,d')$ are strongly isotopic, where $0<d,\,d'\le l-1$. Then there exists linearized permutations $L(x)$, $N(x)$ such that
\begin{equation} \label{3.1}
   L(x\ast_d y)= N(x)\ast_{d'} N(y).
\end{equation}
Write $L(x)=\sum\limits_{i=0}^{2lh-1} a_i x^{p^i}$ and $N(x)=\sum\limits_{i=1}^{2lh-1}b_ix^{p^i}$, with all coefficients in $\F_{q^{2l}}$. Recall that $q=p^h$. We shall do addition and substraction in the subscripts of $a_i$'s and $b_i$'s, and in that case we understand that the subscripts are read modulo $2lh$. Let $C_L$ (resp. $C_R$) be the left (resp. right) hand side of Eqn. \eqref{3.1}. Then we have
\[
   C_L=L\left(\tr(x^{q^l}y)+\tr\big(\beta(x^{q^d}y+xy^{q^d})\big)\, \omega \right),
\]
and
\begin{equation*}
   C_R=\tr\left(N(x)^{q^l}N(y)\right)+ \tr\left(\beta(N(x)^{q^{d'}}N(y)+N(x)N(y)^{q^{d'}}) \right)\, \omega.
\end{equation*}
We deduce that
\begin{align}
   C_L+C_L^{q^l}& = 2 N(x)^{q^l}N(y)+2N(x)N(y)^{q^l}\label{3.2},\\
   C_L-C_L^{q^l}& =  2\,\tr\left(\beta\big(N(x)^{q^{d'}}N(y)+N(x)N(y)^{q^{d'}}\big)\right)\, \omega . \label{3.3}
\end{align}
These two equations hold for all $x,y\in\F_{q^{2l}}$. Therefore, we regard them as polynomial identities in the quotient ring $\F_{q^{2l}}[X,Y]/(X^{q^{2l}}-X,Y^{q^{2l}}-Y)$. In the sequel, we will understand that $x=\overline{X}$, $y=\overline{Y}$ and $\tr(x^{p^i}y^{p^j})=\overline{X}^{p^i}\,\overline{Y}^{p^j}+\overline{X}^{p^iq^l}\,\overline{Y}^{p^jq^l}$, so  it makes sense to talk about the coefficients of the monomials $x^uy^v$, $0\le u,v\le q^{2l}-1$. A key observation is that if $j-i\pmod{2lh}\not\in\{lh,\,d'h,\,(2l-d')h\}$, then $x^{p^i}y^{p^j}$ has zero coefficients on the left hand sides of Eqns. \eqref{3.2}, \eqref{3.3}.
\begin{lemma}\label{lem_3.2}
Suppose that $b_i\ne 0$. Then $b_{j}=0$ and $b_{j+(l-d')h}=0$ if $j-i \pmod{2lh}  \not\in \{0,\,(l+d)h,\,(l-d)h\}$.
\end{lemma}
\begin{proof}
We compare the coefficients of $x^{p^i}y^{p^{i}}$  in Eqns. \eqref{3.2}, \eqref{3.3} and get
\[
     4b_i b_{i+lh}^{q^l}=0,\quad \beta b_{i-d'h}^{q^{d'}}b_{i}+\beta^{q^l}b_{i+(l-d')h}^{q^{d'+l}}\, b_{i+lh}^{q^l}=0.
\]
It follows easily that $b_{i+lh}=0$ and $b_{i-d'h}=0$. Let $j$ be an integer as in the statement of the lemma.
Comparing the coefficients of $x^{p^i}y^{p^{j+lh}}$  in Eqns. \eqref{3.2}, \eqref{3.3}, we get $b_{i}b_{j}^{q^l}=0$ and
$\beta b_{i}b_{j+(l-d')h}^{q^{d'}}+\beta^{q^l}b_{i+(l-d')h}^{q^{l+d'}}b_{j}^{q^l}=0$.
It follows that $b_{j}=0$ and $b_{j+(l-d')h}=0$. This completes the proof.
\end{proof}

\begin{lemma}\label{lem_3.3}
We have $N(x)=b_i x^{p^i}+b_{j} x^{p^{j}}$ for some integers $i,\,j$ such that  $j-i\pmod{2lh} \in \{(l+d)h,\,(l-d)h\}$.
\end{lemma}
\begin{proof}
Assume that $b_i\ne0$. By Lemma \ref{lem_3.2}, $b_j\ne0$ only if $j-i\pmod{2lh}\in  \{0,\,lh+dh,\,lh-dh\}$. It is straightforward to check that $(i+lh+dh)-(i+lh-dh)=2dh\not\in\{0,\,lh+dh,\,lh-dh\}$ by the fact that $\gcd(l,\,d)=1$ and $l+d$ is odd. Hence $b_{i+(l+d)h}$ and $b_{i+(l-d)h}$ can not both be nonzero by the same lemma. This completes the proof.
\end{proof}

\begin{lemma}\label{lem_3.4}
We have $d=d'$, and $N(x)$ is a monomial.
\end{lemma}
\begin{proof}
If $d'\ne d$, then $d'h\not\in\{lh,\,dh,\,(2l-d)h\}$, and by examining the coefficients of $x^{p^iq^{d'}}y^{p^i}$ and $x^{p^jq^{d'}}y^{p^j}$ in Eqn. \eqref{3.3} we see that $\beta b_i^{q^{d'}+1}=\beta b_j^{q^{d'}+1}=0$, so $b_i=b_j=0$ and $N(x)\equiv 0$: a contradiction. This proves that $d'=d$.
By interchanging $i,\,j$ if necessary, we assume that $j=i+(d+l)h\pmod{2lh}$ without loss of generality. As before, we can check that $j+dh-i= (2d+l)h  \pmod{2lh}\not\in \{lh,\,dh,\, (2l-d)h\}$, and the coefficients of ${x^{p^{j+dh}}}y^{p^i}$ in Eqn. \eqref{3.3} yields that $\beta b_j^{q^d}b_i =0$. Hence $N(x)$ is a monomial.
\end{proof}

Here we are in position to complete the proof. Assume that $N(x)=b x^{p^i}$, $0\le i\le 2lh-1$. By expanding \eqref{3.1}, we have
$$L(\tr(x^{q^l}y))+L(\tr(\beta(x^{q^d}y+xy^{q^d}))\omega)= b^{q^l+1} Tr(x^{q^l}y)+ Tr(\beta b^{q^d+1}(x^{q^d}y+xy^{q^d}))\omega.$$
By comparing the coefficients, we see that it holds if and only if
 \[ a_{j}+a_{j+lh}= \left\lbrace
 \begin{array}{ll}
   b^{q^l+1}, & \textup{if } j=i ,\\
   b^{q^l+1}, & \textup{if } j=i+lh, \\
   0 ,& \text{otherwise. }\\
\end{array}\right.    \quad a_j-a_{j+lh}=\left\lbrace \begin{array}{ll}
  b^{q^d+1} ,&\textup{if } j=i ,\\
  -b^{q^{d+l}+q^l} &\textup{if } j=i+lh, \\
  0 , & \text{otherwise. } \\
\end{array} \right.\]
 It follows that $a_j=a_{j+lh}=0$ for $j\neq i \pmod{lh}$,\, $b^{q^d+1}=b^{q^{d+l}+q^l}$ ,i.e. $b^{(q^d+1)(q^l-1)}=1$, and 
\[
L(x)=\frac{1}{2}b^{q^l+1} (x+x^{q^l})^{p^i}+\frac{1}{2}b^{q^d+1}(x-x^{q^l})^{p^i}.
\]
Since $\gcd(q^d+1,q^l+1)=2$ by Lemma \ref{lem_b0b1}, we have $b^2\in\F_{q^l}$. Recall that $\omega^{q^l}= -\omega$, we deduce that $b \in \F_{q^l}^*\cup \F_{q^l}^*\, \omega$. Then the map $x\mapsto N(x)$ is a permutation of $\F_{q^{2l}}$, and $x\mapsto L(x)$ is automatically a permutation since Eqn. \eqref{3.1} holds. Therefore, we get all the desired strong isotopism $(N,\,N,\,L)$'s in this way. This completes the proof of Theorem \ref{thm_str}.

 \section{Isotopism calsses of the Budaghyan-Hellsseth presemifields}
 This section is devoted to the proof of the following main theorem.

\begin{thm}\label{thm_iso}
If $0<d,\,d'<l$, then the presemifields $BH(q,\,l,\,d)$ and $BH(q,\,l,\,d')$ are isotopic if and only  either (i) $d'=d$, or (ii) $q\equiv1\pmod 4$, $d'=l-d$ and $l$ is even.
\end{thm}
We again split the proof into several lemmas.  First we introduce the notation that we shall use throughout this section. Let $\bS_d=(\F_{q^{2l}},+,\star_d)$ be the semifield isotope of $BH(q,\,l,\,d)$ with $\star_d$ as defined in Eqn. \eqref{2.2}. For $x\in\F_{q^{2l}}$,  define $K_d(x)= x\ast_d 1$. Let $*_{d'}$, $\star_{d'}$ and $\bS_{d'}$ and $K_{d'}$ be the corresponding objects for $BH(q,\,l,\,d')$. We recall the following result \cite{Coulter2008Commutative}.
\begin{thm}\label{thm_alpha}\cite[Theorem 2.5]{Coulter2008Commutative}  
Let $\bS_1=(\F_q,\,+,\,\star_1)$ and $\bS_2=(\F_{q},\,+,\,\star_2)$ be isotopic commutative semifields. Then there exists an isotopism  $(M,\,N,\,L)$ between $\bS_1$ and $\bS_2$ such that either
\begin{enumerate}
\item[(i)] $M=N$, or
\item[(ii)] $M(x)\equiv \alpha \star_1 N(x)  \mod(X^q-X)$, where $\alpha \in N_m({\bS_1})$ can not be written as the product of an element of $N(\bS_1)$ and a square of $N_m (\bS_1)$.
\end{enumerate}
\end{thm}

Assume that the two presemifields $BH(q,\,l,\,d)$ and $BH(q,\,l,\,d')$ are isotopic, so that $\bS_d$ and $\bS_{d'}$ are isotopic semifields. If they are strongly isotopic, then we must have $d=d'$ by Theorem \ref{thm_str}. Now assume that the two presemifields are isotopic but not strongly isotopic. By Theorem \ref{thm_alpha}, there exists linearized permutations $L(x),\,N(x)$ over $\F_{q^{2l}}$ and  a nonsquare $\alpha$ in $N_m(\bS_d)$ such that
$\left(N(x) \star_d \alpha\right) \star_d N(y) = L(x \star_{d'} y)$ for all $x,\,y\in\F_{q^{2l}}$. By a change of variables, we can rewrite it as
\begin{equation}\label{eqn_orig}
(K_d(x) \star_d \alpha)\star_d K_d(y) = L\left( N^{-1}(K_d(x)) \star_{d'} N^{-1}(K_d(y)) \right)
\end{equation}
 Recall that $K_{d'}(x) \star_{d'} K_{d'}(y) = x \ast_{d'} y$,  the right hand side of Eqn. \eqref{eqn_orig} is equal to 
$$L\left( K_{d'}^{-1} \big(N^{-1}(K_d(x)) \big) \ast_{d'} K_{d'}^{-1}\big( N^{-1}(K_d(y)) \big) \right).$$ 
Then with $L':=L^{-1}$ and $N':=K_{d'}^{-1}N^{-1}K_d$, the equation now takes the  form
\[
L'\big( (K_d(x) \star_d \alpha) \star_d K_d(y) \big)  = N'(x) \ast_{d'} N'(y).
\]
By Theorem \ref{thm_Nm}, there exists $a,\,b\in\F_q$ such that $\alpha=\kappa(a,b)$ and $K_d(x)\star_d \alpha = K_d(ax+b\xi x^{q^l})$, where $\xi$ is a constant such that $\xi^{q^{l+d}-1}=\beta^{1-q^l}$. By Lemma \ref{lemma_bne0}, $b\ne 0$. Plugging this into the above equation, we get
\begin{equation} \label{4.1}
  L'(a\cdot x\ast_d y)+L'(b\cdot(\xi x^{q^l})\ast_d y) = N'(x) \ast_{d'} N'(y)
\end{equation}
We denote the left (resp. right) hand side of Eqn. \eqref{4.1} by $C_L$ (resp. $C_R$). By expanding we get  the expression for $C_L$ as follows:
\begin{align*}
L'\left( a\, \tr(x^{q^l}y)\right)+L'\left(b\, \tr(\xi^{q^l} xy)\right)+  L'\left(a\, \tr\big(\beta(x^{q^d}y+xy^{q^d}\big)\,\omega\right)+ L'\left(b\,\tr\big(\beta\xi (x^{q^d}y^{q^l}+ x^{q^l}y^{q^d})\big)\,\omega\right).
\end{align*}
In the last term, we have used the fact that  $\xi^{q^{l+d}-1}=\beta^{1-q^l}$.
By the definition of $\ast_{d'}$,
\begin{equation*}
C_R=\tr\left(N'(x)N'(y)^{q^l}\right)+ \tr\left(\beta(N'(x)^{q^{d'}}N'(y)+N'(x)N'(y)^{q^{d'}})\right)\, \omega.
\end{equation*}
From these expressions we compute that
\begin{align}
C_L+C_L^{q^l}&= 2\, N'(x)^{q^l}N'(y)+2\, N'(x)N'(y)^{q^l},\label{4.2}\\
C_L-C_L^{q^l}&=2\,\tr\left(\beta(N'(x)^{q^{d'}}N'(y)+N'(x)N'(y)^{q^{d'}})\right)\, \omega. \label{4.3}
\end{align}
Since $L'$ and $N'$ are linearized, we have
\[
L'(x)=\sum_{i=0}^{2lh-1}a_ix^{p^i},\quad N'(x)=\sum_{i=0}^{2lh-1}b_ix^{p^i},
\]
for some constants $a_i$'s and $b_i$'s in $\F_{q^{2l}}$. And the subscripts of $a_i$'s and $b_i
$'s are read modulo $2lh$ as in Section 3. By equating the coefficients of chosen monomials on both sides of Eqns. \eqref{4.2}, \eqref{4.3}, we can deduce restrictions on the coefficients of $L'$ and $N'$.
\begin{lemma} \label{lem_gamma}
The integer $l$ is even, and $d$, $d'$ are both odd.
If $b_i \ne 0$ for some $i$, then there exists $\gamma_i \in \F_{q^l}^*$ such that  $\gamma_i^{p^{-i}}$ is a root of the quadratic polynomial $ \xi ^{q^l+1}X^2-2(ab^{-1}) X +1\in\F_q[X]$ and $b_{i+lh}= b_i\xi^{p^i}\gamma_i\ne 0$ .
 \end{lemma}
\begin{proof}
Since $N'$ is a permutation, at least one of its coefficient is nonzero. Assume that $b_i\ne 0$, $0\le i\le 2lh-1$. We compare the coefficients of $x^{p^i}y^{p^i}$, $x^{p^{i}q^l}y^{p^i}$ and $x^{p^{i}q^l}y^{p^{i}q^l}$ respectively, on both sides of  Eqn. \eqref{4.2}, and get
\begin{align}
   &\tr(a_{i}+a_{i+lh})\ (b\xi^{q^l})^{p^i}  =  4b_{i+lh}^{q^l}b_{i}\label{eqn_i1}\\
   &\tr(a_{i}+a_{i+lh})\ a^{p^i} = 2 b_{i}^{q^l+1}+2b_{i+lh}^{q^l+1}\label{eqn_i2} \\
   &\tr(a_{i}+a_{i+lh})\ (b\xi)^{p^i}  = 4b_{i}^{q^l}b_{i+lh}\label{eqn_i3}.
\end{align}
The right hand sides can not all be zero, so $\tr(a_{i}+a_{i+lh})\ne 0$. Since $b\ne 0$ by Lemma \ref{lemma_bne0}, Eqns. \eqref{eqn_i1} \eqref{eqn_i3} show that $b_{i+lh}\ne 0$ and $b_{i+lh}=b_i\xi^{p^i}\gamma_i$ for some $\gamma_i\in\F_{q^l}^\ast$. This proves the second claim.

From Eqns. \eqref{eqn_i1}-\eqref{eqn_i3} we deduce  that $b_i^{-(q^l+1)}\,\tr(a_{i}+a_{i+lh})$ is equal to both $4b^{-p^i}\gamma_i$ and $2a^{-p^i}(1 +\xi^{p^i(q^l+1)}\gamma_i^2)$. Therefore, $4b^{-p^i}\gamma_i=2a^{-p^i}(1 +\xi^{p^i(q^l+1)}\gamma_i^2)$, which simplifies to  $ \xi^{q^l+1}(\gamma_i^{p^{-i}})^2-2ab^{-1}\gamma_i^{p^{-i}} +1=0$.  The equation $ \xi^{q^l+1}X^2-2ab^{-1} X +1=0$ has coefficients in $\F_{q}$, and its determinant $4((ab^{-1})^2-\xi^{q^l+1})$ is a nonsquare in $\F_{q}$ by Lemma \ref{lemma_bne0}. Since $\gamma_i^{p^{-i}}$ is a solution, we have $\gamma_i\in\F_{q^2}\setminus\F_{q}$. Since $\gamma_i\in\F_{q^l}^*$,  $l$ must be even. Recall that $l+d$ and $l+d'$ are odd, so both $d$ and $d'$ are odd.
\end{proof}

Observe that if $x^{p^i}y^{p^j}$ ($0\le i,\,j\le 2lh-1$) has a nonzero coefficient in $C_L$, then $i-j\pmod{2lh}$ must lie in the set $\{0,\,lh,\,dh,\,2lh-dh,\,dh+lh,\,lh-dh\}$, or equivalently, $i-j\pmod{lh}\in\{0,\,dh,\,lh-dh\}$. The same is true for $C_L^{q^l}$. Therefore, if $i-j\pmod{lh}\not\in\{0,\,dh,\,lh-dh\}$, then the coefficients of $x^{p^i}y^{p^j}$ on the right hand sides of Eqns. \eqref{4.2}, \eqref{4.3} are zero.  That is
\begin{align}
    &b_{i+lh}^{q^l}b_{j}+b_{i}b_{j+lh}^{q^l}=0, \label{eqn_bi1}\\
    &\beta(b_{i-d'h}^{q^{d'}}b_{j}+b_{i}b_{j-d'h}^{q^{d'}})+\beta^{q^l}(b_{i-d'h+lh}^{q^{d'}}b_{j+lh}+b_{i+lh}b_{j-d'h+lh}^{q^{d'}})^{q^l}=0.\label{eqn_bi2}
\end{align}
\begin{lemma}\label{lem_cob}
Let $(i,\,j)$ be a pair such that
\begin{equation}\label{cond_ij}
0\le i,\,j\le 2lh-1,\quad i-j\pmod{lh}\not\in\{0,\,dh,\,lh-dh\}.
\end{equation}
If $b_i\ne 0$, then $b_{j}=0$.
\end{lemma}
\begin{proof}
By the argument preceding this lemma, Eqn. \eqref{eqn_bi1} holds for this pair $(i,\,j)$. Write $i'=i+lh\pmod{2lh}$. Since the pair $(i',\,j)$ also satisfies the condition \eqref{cond_ij}, it holds that
\begin{equation}\label{eqn_bip}
  b_{i}^{q^l}b_{j}+b_{i+lh}b_{j+lh}^{q^l}=0.
\end{equation}

Assume that $b_j\ne 0$. By Lemma \ref{lem_gamma}, there exist $\gamma_i,\,\gamma_j$ in $\F_{q^2}\setminus\F_q$ such that $b_{i+lh}=b_i \xi^{p^i}\gamma_i$ and $b_{j+lh}=b_j \xi^{p^j}\gamma_j$. We plug them into Eqns. \eqref{eqn_bi1} and \eqref{eqn_bip} respectively and get
\begin{equation*}
 b_{i}^{q^l}b_{j}\xi^{p^iq^l}\gamma_i=-b_{i}b_{j}^{q^l}\xi^{p^jq^l}\gamma_j,\quad b_{i}^{q^l}b_{j}=-b_{i}b_{j}^{q^l} \xi^{p^i+p^jq^l}\gamma_i \gamma_j .
\end{equation*}
It follows that $(b_ib_j^{-1})^{q^l-1}$ is equal to both $-\xi^{p^jq^l-p^iq^l}\gamma_j\gamma_i^{-1}$ and $-\xi^{p^i+p^jq^l}\gamma_i\gamma_j$.
By equating the two quantities and simplifying, we deduce that $\xi^{q^l+1}\gamma_i^{2p^{-i}}-1=0$. Meanwhile, $ \xi ^{q^l+1}\gamma_i^{2p^{-i}}-2ab^{-1} \gamma_i^{p^{-i}} +1=0$ by Lemma \ref{lem_gamma}. It follows that $ab^{-1} \gamma_i^{p^{-i}}=1$, so $a\ne 0$ and $\gamma_i^{p^{-i}}\in\F_q$. This contradicts the fact that $\gamma_i$ is in $\F_{q^2}\setminus\F_q$. Therefore, $b_j$ must be zero.
\end{proof}

Let $\Lambda$ be the set $\{0\le i\le 2lh-1:\, b_i\ne 0\}$. By Lemma \ref{lem_gamma}, $i\in\Lambda$ if and only $i+lh\pmod{2lh}$ is in $\Lambda$. For any two distinct elements of $\Lambda$, their difference modulo $lh$ is in $\{0,\,dh,\,lh-dh\}$ by Lemma \ref{lem_cob}. We claim that {\it the set $\{i\pmod{lh}:\,i\in\Lambda\}$ has size at most two}. Otherwise, it contains a three-term arithmetic progression modulo $2lh$ with common difference $dh$. The difference between the two nonadjacent terms is $2dh\pmod{lh}$ and it should be in the set $\{0,\,dh,\,lh-dh\}$. This is the case only if $lh$ divides $2dh$ or $3dh$. Since $l$ is even and $d$ is odd by Lemma \ref{lem_gamma}, we must have $l=2d$. However, it follows from $\gcd(l,d)=1$ that $l=2$ and $d=1$, contradicting the assumption that $l>2$. This proves the claim.

Let $i$ be the minimal element of $\Lambda$. Then $\Lambda\subseteq\{i,\,i+dh,\,i+lh,\,i+dh+lh\}$, and
\begin{equation}\label{4.6}
N'(x)= b_i x^{p^i}+b_{i+lh} x^{p^{i+lh}} +b_{i+dh} x^{p^{i+dh}}+b_{i+dh+lh} x^{p^{i+dh+lh}}.
\end{equation}
Moreover,  there exists $\gamma_j \in \F_{q^2}\setminus \F_q$ such that $b_{j+lh}=b_{j}\xi^{p^j}\gamma_j$ and $\gamma_j^{q+1}=\xi^{-p^j(q^l+1)}$ for $j \in \Lambda$ by Lemma \ref{lem_gamma}. We start with a technical lemma.

\begin{lemma}\label{lem_k}
 There does not exist an odd integer $k$ such that
\begin{equation}\label{eqn_k}
\beta b_i^{q^k+1}+\beta^{q^l}b_{i+lh}^{(q^k+1)q^l}=0,\quad \beta b_{i}^{q^k}b_{i+lh}+\beta^{q^l} b_{i+lh}^{q^{k+l}}b_{i}^{q^l}=0.
\end{equation}
\end{lemma}
\begin{proof}
We plug $b_{i+lh}=b_i\xi^{p^i}\gamma_i$ into the above equations and simplify them to get
\[
(\beta b_i^{q^{k}+1})^{q^l-1}\xi^{p^i(q^{k+l}+q^l)}\gamma_i^{q+1}=-1,\quad
(\beta b_i^{q^{k}+1})^{q^l-1}\xi^{p^i(q^{k+l}-1)}\gamma_i^{q-1}=-1.
\]
Taking quotient, we get $\gamma_i^{2}=\xi^{-p^i(q^l-1)}$. Together with $\gamma_i^{q+1}=\xi^{-p^i(q^l+1)}$, we deduce that $\gamma_i^{q-1}=1$, contradicting the fact that $\gamma_i \notin \F_q$. This completes the proof.
\end{proof}

\begin{lemma}\label{lem_dp}
We have $d'=d$ or $d'=l-d$.
\end{lemma}
\begin{proof}
Suppose to the contrary that $d'\ne d,\,l-d$. Then $i+d'h\not\in\{i,\,i+dh,\,i+lh,\,i+dh+lh\}$, and so $b_{i+d'h}=b_{i+d'h+lh}=0$.
The pairs $(i+d'h,\,i)$, $(i+d'h,\,i+lh)$, with each entry taken modulo $2lh$, satisfy the condition \eqref{cond_ij}, and Eqn. \eqref{eqn_bi2} now takes the form as in Eqn. \eqref{eqn_k} with $k=d'$.  This is impossible by Lemma \ref{lem_k}.
\end{proof}

\begin{lemma}
We have $b_{i+dh}=b_{i+dh+lh}=0$, so $N'(x)= b_i x^{p^i}+b_{i+lh} x^{p^{i+lh}}$.
\end{lemma}
\begin{proof}
By Lemma \ref{lem_gamma}, we just need to show that $b_{i+dh}=0$. Suppose to the contrary that $b_{i+dh}\ne 0$. We have $i+2dh\not\in\{i,\,i+lh,\,i+dh,\,i+dh+lh\}$, so $b_{i+2dh}=b_{i+2dh+lh}=0$. The proof is very similar to that of Lemma \ref{lem_dp}, and we need to split into two cases.

{\bf Case 1: $d'=d$}.  The pairs $(i+2dh,\,i)$, $(i+2dh,\,i+lh)$ with each entry taken modulo $2lh$, satisfy the condition \eqref{cond_ij} and Eqn. \eqref{eqn_bi2} read as
$\beta b_{i+dh}^{q^d}b_{i}+\beta^{q^l} b_{i+dh+lh}^{q^{d+l}}b_{i+lh}^{q^l} =0$,
$\beta b_{i+dh}^{q^d}b_{i+lh}+\beta^{q^l} b_{i+dh+lh}^{q^{d+l}}b_{i}^{q^l}=0$.
As in the proof of Lemma \ref{lem_dp}, we deduce that
\begin{equation*}
(\beta b_{i+dh}^{q^d}b_{i})^{q^l-1} \gamma_i \gamma_{i+dh}^q \xi^{p^{i}(q^{l+2d}+q^l)}=-1,\quad
(\beta b_{i+dh}^{q^d}b_{i})^{q^l-1} \gamma_i^{-1}\gamma_{i+dh}^q \xi^{p^{i}(q^{l+2d}-1)}=-1
\end{equation*}
The rest of the argument is the same as in that of Lemma \ref{lem_dp}.

{\bf Case 2: $d'=l-d$}. In this case, by considering the pairs $(i+d'h,\,i+dh)$ and $(i+d'h,\,i+dh+lh)$, the same argument exactly leads to the contradiction that $\gamma_{i+dh}\in\F_q$, while it should be that $\gamma_{i+dh}\in\F_{q^2}\setminus\F_q$.  This completes the proof.
\end{proof}

\begin{lemma}\label{lem_lmd}
We have $d'=l-d$.
\end{lemma}
\begin{proof}
Assume to the contrary that $d=d'$.   We compare the coefficients of $x^{p^{i+dh}}y^{p^{i}}$ and $x^{p^{i+dh}}y^{p^{i+lh}}$ in Eqn.\eqref{4.1} and get the equations
\begin{equation*}
\begin{array}{c}
(a_{i}-a_{i+lh})\ (a\beta \omega)^{p^i} = (\beta b_{i}^{q^d+1}+\beta^{q^l} b_{i+lh}^{q^{d+l}+q^l})\ \omega, \\
 (a_{i}-a_{i+lh})\ (b\beta \xi\omega)^{p^i} = (\beta b_{i}^{q^d} b_{i+lh}+\beta^{q^l} b_{i+lh}^{q^{d+l}} b_{i}^{q^l})\ \omega.
 \end{array}
\end{equation*}

If $a_i=a_{i+lh}$, the we would get a contradiction to Lemma \ref{lem_k} with $k=d$. Hence $a_i-a_{i+lh}\ne 0$. We now take quotient of both sides and plug in $b_{i+h}=b_i\xi^{p^i}\gamma_i$ to get
\[
(ab^{-1}  \xi^{-1} )^{p^i}=\frac{1 +(\beta b_i^{q^d+1})^{q^l-1}\xi^{p^i(q^{d+l}-1)}}
{ \xi^{p^i}\gamma_i+(\beta b_i^{q^d+1})^{q^l-1}\xi^{p^iq^{d+l}}\gamma_i^{q}}.
\]
Here we have used the fact that $\gamma_i^{q+1}=\xi^{-p^i(q^l+1)}$. With $t:=(\beta b_i^{q^d+1})^{q^l-1}\xi^{p^i(q^{l+d}-1)}$, we can rewrite it as $((ab^{-1})^{p^i}\gamma_i^q-1)\ t=-(ab^{-1})^{p^i}\gamma_i+1$. Recall that $\xi^{q^l+1}\in\F_q$, $\gamma_i\in\F_{q^2}\setminus\F_q$ and $l$ is even. Raising both sides to the $(q^l+1)$-st power, we deduce that $(1-\gamma_i^q(ab^{-1})^{p^i})^2-(1-\gamma_i(ab^{-1})^{p^i})^2=0$. It follows that  $(ab^{-1})^{p^i}\ (2-(ab^{-1})^{p^i}(\gamma_i+\gamma_i^q))=0$.

By Lemma \ref{lem_gamma}, we have $\gamma_i+\gamma_i^q=2(ab^{-1})^{p^i}\xi^{-p^i(q^l+1)}$, so
\[
2-(ab^{-1})^{p^i}\ (\gamma_i+\gamma_i^q)=2-2(ab^{-1})^{2p^i}\xi^{-p^i(q^l+1)}
=-2b^{-2p^i}\xi^{-p^i(q^l+1)}(a^2-b^2\xi^{q^l+1})^{p^i}.
\]
It is nonzero since $a^2-b^2\xi^{q^l+1}$ is a nonsquare in $\F_{q}$.  We thus must have $a=0$, and so $t=-1$. Recall that $\beta^{q^l-1}=\xi^{1-q^{l+d}}$, so it gives that $b_i^{(q^d+1)(q^l-1)}\xi^{(p^i-1)(q^{l+d}-1)}=-1$. Raising it to $\frac{q^l+1}{2}$-th power, the left hand side yields $1$. Since $l$ is even, we have $q^l\equiv 1\pmod{4}$ and $\frac{q^l+1}{2}$ is odd, so the right hand side remains $-1$: a contradiction. This completes the proof.
\end{proof}

We are now in a position to complete the proof. By comparing the coefficients of  $x^{p^iq^{l-d}}y^{p^iq^l}$ and $x^{p^iq^{2l-d}}y^{p^iq^l}$ in Eqn. \eqref{4.1}, we get
\begin{align*}
(a_{i+lh-dh}-a_{i+2lh-dh})\ (a\beta\omega)^{p^iq^{l-d}}&= (\beta b_i^{q^{l-d}}b_{i+lh}+\beta^{q^l}b_{i+lh}^{q^{2l-d}}b_{i}^{q^l})\ \omega, \\
(a_{i+lh-dh}-a_{i+2lh-dh})\ (b\beta\xi\omega)^{p^iq^{l-d}}&= (\beta b_{i+lh}^{q^{l-d}+1}+\beta^{q^l}b_{i}^{q^{2l-d}+q^l})\ \omega.
\end{align*}
By exactly the same argument as in the proof of Lemma \ref{lem_lmd}, we deduce that $a=0$.  The fact $a^2-b^2\xi^{q^l+1}$ is a nonsquare in $\F_{q}$ gives that $(-\xi^{q^l+1})^{(q-1)/2}=-1$, i.e., $\xi^{(q^l+1)(q-1)/2}=(-1)^{(q+1)/2}$. Recall that  $\xi^{q^{l+d}-1}=\beta^{1-q^l}$, $\beta$ is a nonsquare in $\F_{q^{2l}}$, so
\begin{align*}
-1=\beta^{(q^{2l}-1)/2}=\xi^{(q^l+1)(q^{l+d}-1)/2}=(-1)^{(q+1)(l+d)/2}=(-1)^{(q+1)/2}.
\end{align*}
Here we have used the fact that $l+d$ is odd. We thus conclude that $q\equiv 1\pmod{4}$.

We now explicitly construct an isotopism of the desired form between $BH(q,\,l,\,d)$ and $BH(q,\,l,\,l-d)$ in the case $q\equiv 1\pmod{4}$ and $l$ is an even integer larger than $2$. In this case,  $\omega$ is a nonsquare in $\F_{q^{2l}}$, where $\omega$ is as in Eqn. \eqref{2.1} with $\omega+\omega^{q^l}=0$. By Lemma \ref{lem_2.2}, we set $\beta:=w^{-1}$ without changing the isotopism class.
Take $\xi\in\F_{q^2}^*$ such that $\xi^q+\xi=0$. Then $\beta^{q^l-1}=\xi^{1-q^{l+d}}=-1$. We now set $a=0$, $b=1$, $\alpha=\kappa(a,b)$, and
\begin{align*}
L'(x)= \xi^{(q-3)/2}(x+x^{q^l})+\xi^{-1}(x-x^{q^l})^{q^{l-d}},\,\quad
N'(x)= x + \xi^{(q-1)/2} x^{q^l}	
\end{align*}
It is straightforward to check that $L'$ and $N'$ are both permutations over $\F_{q^{2l}}$, $a^2-b^2\xi^{q^l+1}=-\xi^{q^l+1}$ is a nonsquare of $\F_q$, and Eqn. \eqref{4.1} holds. From $L'=L^{-1}$ and $N'=K_{d'}^{-1}N^{-1}K_d$ we can reconstruct $L$ and $N$ such that $(N,\,\alpha\star_d N,\,L)$ is an isotopism between the semifields $\bS_d$ and $\bS_{l-d}$. This completes the proof of Theorem \ref{thm_iso}.
\vspace*{3mm}

Since the autotopisms of $BH(q,\,l,\,d)$ must be strong autotopism by the proof of Theorem \ref{thm_iso}, we deduce from Theorem \ref{thm_Nm} and Theorem \ref{thm_str} that the autotopism group of the presemifield $BH(q,\,l,\,d)$ has order $2lh(q^{l}-1)(q^2-1)$. It also follows from Theorem \ref{thm_iso} that, for fixed $q$ and $l>2$, the number of isotopism classes in the Budaghyan-Helleseth family is a half of the size of
$\{0<d<l:\,\gcd(l,d)=1\}$ in the case $q\equiv1\pmod{4}$ and $l$ is even, and is equal to the size of
\[
\{0<d<l:\,\gcd(l,d)=1,\quad\textup{$l+d$ is odd}\}
\]
in all the other cases. That is,
\begin{enumerate}
\item[(1)] in the case $q\equiv1\pmod{4}$ and $l$ is even, the number is $\phi(l)/2$;
\item[(2)] in the case $q\equiv3\pmod{4}$ and $l$ is even, the number is $\phi(l)$;
\item[(3)] in the case $l$ is odd,  the number is $\phi(l)/2$, since exactly one of $l-d$ and $d$ is even.
\end{enumerate}
Here, $\phi$ is the Euler totient function.

\vspace*{3mm}
\noindent\textbf{Acknowledgement}. 

This work was supported by National Natural Science Foundation of China under Grant No. 11771392. The authors would like to thank Dr. Longjiang Qu for discussions that led to this project.

\section*{Reference}
\scriptsize 
\setlength{\bibsep}{0.5ex}  

 \bibliographystyle{plain}

\end{document}